\documentclass[british,a4paper]{scrartcl}
\usepackage[T1]{fontenc}
\usepackage[latin9]{inputenc}
\usepackage[a4paper]{geometry}
\usepackage{babel}
\usepackage{prettyref}
\usepackage{mathtools}
\usepackage{amsmath}
\usepackage{amsthm}
\usepackage{amssymb}
\usepackage{mathdots}
\usepackage[unicode=true,pdfusetitle,
 bookmarks=true,bookmarksnumbered=false,bookmarksopen=false,
 breaklinks=false,pdfborder={0 0 1},backref=false,colorlinks=false]
 {hyperref}

\makeatletter
\theoremstyle{definition}
\newtheorem*{defn*}{\protect\definitionname}
\theoremstyle{plain}
\newtheorem{thm}{\protect\theoremname}[section]
\theoremstyle{remark}
\newtheorem{rem}[thm]{\protect\remarkname}
\theoremstyle{plain}
\newtheorem{lem}[thm]{\protect\lemmaname}
\theoremstyle{definition}
\newtheorem{example}[thm]{\protect\examplename}
\theoremstyle{plain}
\newtheorem{cor}[thm]{\protect\corollaryname}
\theoremstyle{plain}
\newtheorem{prop}[thm]{\protect\propositionname}

\@ifundefined{date}{}{\date{}}
\usepackage{prettyref}

\usepackage{enumerate}
\allowdisplaybreaks

\newcommand{\lin}{\operatorname{lin}}
\newcommand*{\e}{\mathrm{e}}

\newcommand{\R}{\mathbb{R}}
\newcommand{\N}{\mathbb{N}}

\newcommand{\dom}{\operatorname{dom}}
\newcommand{\ran}{\operatorname{ran}}

\renewcommand{\d}{\,\mathrm{d}}

\renewcommand{\tilde}{\widetilde}

\newcommand{\tr}{\operatorname{tr}}

\newrefformat{subsec}{Subsection \ref{#1}}
\newrefformat{prob}{Problem \ref{#1}}
\newrefformat{prop}{Proposition \ref{#1}}
\newrefformat{lem}{Lemma \ref{#1}}
\newrefformat{thm}{Theorem \ref{#1}}
\newrefformat{cor}{Corollary \ref{#1}}
\newrefformat{rem}{Remark \ref{#1}}
\newrefformat{exa}{Example \ref{#1}}
\newrefformat{sub}{Subsection \ref{#1}}
\newrefformat{eq}{(\ref{#1})}

\theoremstyle{definition}

\newrefformat{hyp}{Hypotheses \ref{#1}}

\makeatother

\providecommand{\corollaryname}{Corollary}
\providecommand{\definitionname}{Definition}
\providecommand{\examplename}{Example}
\providecommand{\lemmaname}{Lemma}
\providecommand{\propositionname}{Proposition}
\providecommand{\remarkname}{Remark}
\providecommand{\theoremname}{Theorem}

\begin{document}
\title{M-Accretivity via Boundary Systems}
\author{Sascha Trostorff\thanks{Mathematisches Seminar, CAU Kiel, Germany}}
\maketitle
\begin{abstract}
\textbf{Abstract.} We consider skew-symmetric operators on a Hilbert
space and study m-accretive restrictions of their negative adjoints.
Using the theory of boundary systems, we provide a full characterisation
of all those m-accretive restrictions, linear and nonlinear ones.
The result is then applied to port-Hamiltonian systems of arbitrary
order.
\end{abstract}
\textbf{Keywords.} M-accretive Operators, Boundary Systems, Port-Hamiltonian
Systems\\
\label{=00005Cnoindent}\textbf{MSC2020. }47H06,\textbf{ }47N20.

\section{Introduction}

M-accretive operators on Hilbert spaces play a fundamental role in
the theory of partial differential equations. In the linear case,
the negative of m-accretive operators are precisely the generators
of contraction semigroups by the famous Lumer-Phillips Theorem (see
e.g. \cite[Chapter II, Theorem 3.15]{Engel_Nagel2000}); that is,
for an m-accretive operator $A$, the corresponding linear Cauchy
problem 
\[
u'(t)=-Au(t),\quad u(0)=u_{0}
\]
is well-posed for each $u_{0}\in\overline{\dom}(A)=H$, where $H$
denotes the underlying Hilbert space. This theory generalises directly
to the nonlinear case, where m-accretive operators are precisely the
negative generators of nonlinear contraction semigroups. This result
was first proved by Komura in \cite{Komura1967}. That means, that
also in the case of a nonlinear m-accretive operator $A$, the nonlinear
Cauchy problem 
\[
u'(t)=-A(u(t)),\quad u(0)=u_{0}
\]
is well-posed for each $u_{0}\in\overline{\dom}(A).$ Hence, one way
to obtain the well-posedness of Cauchy problems is to prove the m-accretivity
of the spatial differential operator. We remark here, that m-accretivity
does not only imply well-posedness of Cauchy problems, but plays also
a crucial role in more general differential equations (linear and
nonlinear) and build a cornerstone in the theory of evolutionary equations
and inclusions (see \cite{STW2022}).

In order to prove that a certain differential operator is m-accretive,
boundary conditions come into play. Indeed, in most cases, the differential
operator is indeed linear and the nonlinearity occurs in the imposed
boundary condition; i.e., the differential operator with boundary
condition is a restriction of a linear differential operator without
boundary condition. Thus, a natural question arises: Which boundary
conditions lead to m-accretive operators? In this article, we give
an answer to this question for restrictions of negative adjoints of
skew-symmetric operators by employing the theory of boundary systems.
From an application point of view, the consideration of skew-symmetric
operators is quite natural, as most (if not all) equations in mathematical
physics can be written as suitable systems with skew-symmetric spatial
differential operators. This is the key observation behind the theory
of evolutionary equations and goes to back to Picard, \cite{Picard2009}. 

Boundary systems were first introduced by Schubert et al. in \cite{SSVW2015}
and provide a generalisation of the well established theory of boundary
triplets (see e.g. \cite[Section 3.1]{Gorbachuk1991} or \cite[Section 14]{Schmudgen2012}).
In the particular situation of boundary triplets for skew-symmetric
operators we refer to \cite{Wegner2017}. A comparison of both concepts,
boundary systems and boundary triplets, was provided in \cite{Waurick_Wegner_2018},
where it was shown that a skew-symmetric operator has always a boundary
system but a boundary triplet only if it has a skew-selfadjoint extension.
Hence, boundary systems are indeed a more general concept than boundary
triplets. Both, boundary triplets and boundary systems were successfully
employed to prove characterisation results for \textbf{linear} m-accretive
extensions of skew-symmetric operators (which are automatically restrictions
of the negative adjoint, see e.g. \cite[Proposition 2.8]{Wegner2017}
or \prettyref{cor:extension_vs_restriction} below). 

In this article we show, that the theory of boundary systems can also
be used to prove a characterisation result for \textbf{all} (linear
and nonlinear) m-accretive realisations (these are in general no extensions
of the skew-symmetric operator but restrictions of its negative adjoint).
The proof relies on a recent result by Picard and the author, \cite{Picard_Trostorff_2023},
where a concrete boundary system is used to characterise all m-accretive
realisations of a skew-symmetric operator. We use this characterisation
to generalise the statement to any boundary system associated with
the skew-symmetric operator. 

The article is structured as follows: We begin to recall the notions
of m-accretive operators, boundary systems for skew-symmetric operators
and restate the main result of \cite{Picard_Trostorff_2023}. Then,
we formulate and prove our main result in Section 3. Section 4 is
devoted to an application of the theory developed before. We will
provide a characterisation of all boundary conditions for port-Hamiltonian
system of arbitrary order, yielding an m-accretive realisation of
the port-Hamiltonian operator. For the theory of port-Hamiltonian
systems we refer to the monograph \cite{JacobZwart2012} and to the
PhD theses \cite{Villegas2007,Augner2016}. We emphasise that in the
case of linear realisations, the characterisation result was known
before, but to our best knowledge, a similar result for nonlinear
realisations does not exist in the literature.

Throughout, all Hilbert spaces are assumed to be real. However, the
results can easily be transferred to the complex case by passing to
the realification; that is, by considering the space as a real space
and equip it with the real part of the complex inner product as the
new inner product. 

\section{Preliminaries}

This section is devoted to collect some results, which will be needed
in the subsequent sections. Throughout, $H$ denotes a Hilbert space
with inner product $\langle\cdot,\cdot\rangle$. We begin with the
definition of m-accretive operators.
\begin{defn*}
A mapping $B\colon\dom(B)\subseteq H\to H$ is called \textbf{accretive},
if for all $u,v\in\dom(B)$
\[
\langle B(u)-B(v),u-v\rangle\geq0.
\]
$B$ is called \textbf{m-accretive}, if $B$ is accretive and $1+B$
is onto.
\end{defn*}
\begin{rem}
\label{rem:m-accretive}\begin{enumerate}[(a)]

\item If $B$ is m-accretive and $A\colon\dom(A)\subseteq H\to H$
is an accretive extension of $B,$ then $A=B.$ Indeed, for $x\in\dom(A)$
we find an element $y\in\dom(B)$ such that $y+B(y)=x+A(x)$. Since
$A$ extends $B$, we infer that 
\[
-\|x-y\|^{2}=\langle A(x)-B(y),x-y\rangle=\langle A(x)-A(y),x-y\rangle\geq0,
\]
which yields $x=y\in\dom(B)$ and thus, $A=B.$ 

\item The notions of accretivity and m-accretivity can also be defined
for binary relations $B\subseteq H\times H$. Then it turns out, that
m-accretive relations are precisely the maximal (with respect to inclusion)
accretive relations. This is the celebrated Theorem of Minty, \cite{Minty1962}
(see also \cite[Theorem 17,1,7]{STW2022}) .

\end{enumerate}
\end{rem}

We state and prove a result for m-accretive operators, which will
be useful in the subsequent parts. For a Lipschitz-continuous mapping
$f$ we denote by $|f|_{\mathrm{Lip}}$ the smallest Lipschitz-constant.
\begin{lem}
\label{lem:onto_m_accretive}Let $B\colon\dom(B)\subseteq H\to H$.
Then the following statements are equivalent:

\begin{enumerate}[(i)]

\item $B$ is m-accretive,

\item $B$ is accretive and there exists $\lambda>0$ such that $\lambda+B$
is onto,

\item For each $\mu>0$ the operator $\mu+B$ is bijective and its
inverse is Lipschitz-continuous with $|(\mu+B)^{-1}|_{\mathrm{Lip}}\leq\frac{1}{\mu}$. 

\end{enumerate}
\end{lem}

\begin{proof}
(i) $\Rightarrow$ (ii): That is obvious.\\
(ii) $\Rightarrow$ (iii): Since $B$ is accretive, we have 
\[
\|(\mu+B)(u)-(\mu+B)(v)\|\|u-v\|\geq\langle(\mu+B)(u)-(\mu+B)(v),u-v\rangle\geq\mu\|u-v\|^{2}\quad(u,v\in\dom(B))
\]
and hence, $\mu+B$ is injective for each $\mu>0$ and $(\mu+B)^{-1}\colon\ran(\mu+B)\to H$
is Lipschitz-continuous with $|(\mu+B)^{-1}|_{\mathrm{Lip}}\leq\frac{1}{\mu}$.
It remains to prove that $\mu+B$ is onto. For $\mu=\lambda$ this
is satisfied by assumption. \\
First, let $0<\mu<\lambda$. Then, for $f\in H$ the mapping 
\[
H\ni u\mapsto(\lambda+B)^{-1}(f+(\lambda-\mu)u)
\]
is a strict contraction and hence, possesses a fixed point by the
contraction principle. This fixed point then satisfies $\mu u+B(u)=f$
and thus, $\mu+B$ is onto. Moreover, note that the same argument
applies for $\lambda<\mu<2\lambda$ and hence, by a bootstrapping
argument, we infer that $\mu+B$ is onto for each $\mu>0$.\\
(iii) $\Rightarrow$(i): It suffices to show that $B$ is accretive,
since $1+B$ is onto by assumption. Let $u,v\in\dom(B).$ Then we
compute 
\begin{align*}
\|u-v\|^{2} & =\|(\mu+B)^{-1}(\mu+B)(u)-(\mu+B)^{-1}(\mu+B)(v)\|^{2}\\
 & \leq\frac{1}{\mu^{2}}\|\mu(u-v)+B(u)-B(v)\|^{2}\\
 & =\frac{1}{\mu^{2}}\left(\mu^{2}\|u-v\|^{2}+2\mu\langle B(u)-B(v),u-v\rangle+\|B(u)-B(v)\|^{2}\right)
\end{align*}
and hence, 
\[
0\leq2\langle B(u)-B(v),u-v\rangle+\frac{1}{\mu}\|B(u)-B(v)\|^{2}
\]
for each $\mu>0.$ Letting $\mu\to\infty,$we infer that $B$ is accretive. 
\end{proof}
The other crucial tool for our result are boundary systems for skew-symmetric
operators. We begin to introduce the basic setting. 
\begin{defn*}
Let $A_{0}\colon\dom(A_{0})\subseteq H\to H$ be a skew-symmetric
closed linear operator and set $A\coloneqq-A_{0}^{\ast}.$ We call
a triple $(F,G_{1},G_{2})$ a \textbf{boundary system for $A$},\textbf{
}if $G_{1},G_{2}$ are Hilbert spaces, $F\colon\dom(A)\to G_{1}\times G_{2}$
is linear and onto and 
\[
\forall u,v\in\dom(A):\,\langle Au,v\rangle+\langle u,Av\rangle=\langle F_{1}u,F_{1}v\rangle_{G_{1}}-\langle F_{2}u,F_{2}v\rangle_{G_{2}},
\]
where $F_{1}$ and $F_{2}$ are the components of $F$ with values
in $G_{1}$ and $G_{2}$, respectively.
\end{defn*}
The notion of boundary systems is motivated by the following example.
\begin{example}
Let $a,b\in\R$ with $a<b$ and define $A_{0}\colon H_{0}^{1}(a,b)\subseteq L_{2}(a,b)\to L_{2}(a,b)$
by $A_{0}u=u'.$ Here $H_{0}^{1}(a,b)$ denotes the classical Sobolev
space with vanishing boundary values; that is,
\[
H_{0}^{1}(a,b)=\{u\in H^{1}(a,b)\,;\,u(a)=u(b)=0\}.
\]
An easy computation shows that $A\coloneqq-A_{0}^{\ast}$ is given
by $Au=u'$ with $\dom(A)=H^{1}(a,b).$ Consequently, integration
by parts yields 
\[
\langle Au,v\rangle+\langle u,Av\rangle=\langle u',v\rangle+\langle u,v'\rangle=u(b)v(b)-u(a)v(a)\quad(u,v\in\dom(A)).
\]
Thus, if we set $G_{1}=G_{2}\coloneqq\R$ and $F\colon\dom(A)\to\R\times\R$
with $Fu=(u(b),u(a))$, we obtain a boundary system for $A$. Note
that $F$ is onto as for $x,y\in\R$ we can define 
\[
u(t)\coloneqq\frac{(b-t)y+(t-a)x}{b-a}\quad(t\in[a,b])
\]
and obtain $Fu=(x,y).$ 
\end{example}

Let us collect some simple properties of boundary systems. Note that
$\dom(A)$ is a Hilbert space itself, if we equip it with the inner
product 
\[
\langle u,v\rangle_{A}\coloneqq\langle u,v\rangle+\langle Au,Av\rangle\quad(u,v\in\dom(A)).
\]
We will regard $\dom(A)$ as this Hilbert space without further mentioning. 
\begin{lem}
\label{lem:bd_systems_easy}Let $A_{0}\colon\dom(A_{0})\subseteq H\to H$
be skew-symmetric and closed and set $A\coloneqq-A_{0}^{\ast}.$ Moreover,
let $(F,G_{1},G_{2})$ be a boundary system for $A$. Then $\ker F=\dom(A_{0})$
and $F$ is bounded. 
\end{lem}

\begin{proof}
Let $u\in\dom(A).$ Then using that $F$ is onto, we obtain 
\begin{align*}
u\in\ker F & \Leftrightarrow\forall v\in\dom(A):\:\langle(F_{1}u,-F_{2}u),F(v)\rangle_{G_{1}\times G_{2}}=0\\
 & \Leftrightarrow\forall v\in\dom(A):\:\langle Au,v\rangle+\langle u,Av\rangle=0,\\
 & \Leftrightarrow u\in\dom(A^{\ast})=\dom(A_{0}).
\end{align*}
To prove the boundedness of $F$, it suffices to show that $F$ is
closed by the closed graph theorem. For doing so, let $(u_{n})_{n}$
in $\dom(A)$ with $u_{n}\to u$ with respect to the graph norm of
$A$ and $Fu_{n}\to g$ in $G_{1}\times G_{2}$. Then for each $v\in\dom(A)$
we have 
\begin{align*}
\langle F_{1}u,F_{1}v\rangle_{G_{1}}-\langle F_{2}u,F_{2}v\rangle_{G_{2}} & =\langle Au,v\rangle+\langle u,Av\rangle\\
 & =\lim_{n\to\infty}\langle Au_{n},v\rangle+\langle u_{n},Av\rangle\\
 & =\lim_{n\to\infty}\langle F_{1}u_{n},F_{1}v\rangle_{G_{1}}-\langle F_{2}u_{n},F_{2}v\rangle_{G_{2}}\\
 & =\langle g_{1},F_{1}v\rangle_{G_{1}}-\langle g_{2},F_{2}v\rangle_{G_{2}}\\
 & =\langle(g_{1},-g_{2}),Fv\rangle_{G_{1}\times G_{2}}.
\end{align*}
By the surjectivity of $F$ we infer that $(F_{1}u,-F_{2}u)=(g_{1},-g_{2})$;
i.e., $Fu=g$, which shows the closedness of $F$. 
\end{proof}
As it turns out, each adjoint of a skew-symmetric operator possesses
a boundary system.
\begin{lem}
\label{lem:standard_bd_sys}Let $A_{0}\colon\dom(A_{0})\subseteq H\to H$
be a skew-symmetric closed linear operator and set $A\coloneqq-A_{0}^{\ast}.$
Then $\ker(1-A)$ and $\ker(1+A)$ are closed subspaces of $\dom(A)$.
Moreover, setting 
\[
\pi_{1}\colon\dom(A)\to\ker(1-A),\quad\pi_{-1}\colon\dom(A)\to\ker(1+A)
\]
as the corresponding projections, we infer that $\left((\pi_{1},\pi_{-1}),\ker(1-A),\ker(1+A)\right)$
is a boundary system for $A$, if we equip $\ker(1-A)$ and $\ker(1+A)$
with the graph norm of $A$. 
\end{lem}

\begin{proof}
It is obvious that $\ker(1-A)$ and $\ker(1+A)$ are closed subspaces
of $\dom(A).$ Moreover, $\dom(A_{0})$ is a closed subspace of $\dom(A)$
and for its orthogonal complement, we obtain 
\begin{align*}
u\in\dom(A_{0})^{\bot_{\dom(A)}} & \Leftrightarrow\forall v\in\dom(A_{0}):\langle u,v\rangle_{A}=0,\\
 & \Leftrightarrow\forall v\in\dom(A_{0}):\langle Au,A_{0}v\rangle=-\langle u,v\rangle,\\
 & \Leftrightarrow Au\in\dom(A_{0}^{\ast}),\quad A_{0}^{\ast}Au=-u,\\
 & \Leftrightarrow u\in\ker(1-A^{2})
\end{align*}
and hence, $\dom(A)=\dom(A_{0})\oplus_{A}\ker(1-A^{2}).$ Moreover,
$\ker(1-A^{2})=\ker(1-A)\oplus_{A}\ker(1+A),$ since for $u\in\ker(1-A^{2})$
we infer that 
\[
u=\frac{1}{2}(u+Au)+\frac{1}{2}(u-Au)\in\ker(1-A)+\ker(1+A)
\]
and for $u\in\ker(1-A),v\in\ker(1+A)$ we compute 
\[
\langle u,v\rangle_{A}=\langle Au,Av\rangle+\langle u,v\rangle=\langle u,-v\rangle+\langle u,v\rangle=0.
\]
We denote the orthogonal projector onto $\dom(A_{0})$ by $\pi_{0}$.
Hence, we have for $u\in\dom(A)$
\[
u=\pi_{0}u+\pi_{1}u+\pi_{-1}u\in\dom(A_{0})+\ker(1-A)+\ker(1+A).
\]
It is clear, that $(\pi_{1},\pi_{-1})\colon\dom(A)\to\ker(1-A)\times\ker(1+A)$
is onto and for $v,u\in\dom(A)$ we compute 
\begin{align*}
\langle Au,v\rangle+\langle u,Av\rangle & =\langle A_{0}\pi_{0}u,v\rangle+\langle\pi_{1}u,v\rangle-\langle\pi_{-1}u,v\rangle+\langle u,Av\rangle\\
 & =\langle\pi_{0}u,-Av\rangle+\langle\pi_{1}u,v\rangle-\langle\pi_{-1}u,v\rangle+\langle u,Av\rangle\\
 & =\langle\pi_{1}u,v\rangle-\langle\pi_{-1}u,v\rangle+\langle\pi_{1}u,Av\rangle+\langle\pi_{-1}u,Av\rangle\\
 & =\langle\pi_{1}u,v\rangle-\langle\pi_{-1}u,v\rangle+\langle\pi_{1}u,A_{0}\pi_{0}v\rangle+\langle\pi_{1}u,\pi_{1}v\rangle-\\
 & \quad-\langle\pi_{1}u,\pi_{-1}v\rangle+\langle\pi_{-1}u,A_{0}\pi_{0}v\rangle+\langle\pi_{-1}u,\pi_{1}v\rangle-\langle\pi_{-1}u,\pi_{-1}v\rangle\\
 & =\langle\pi_{1}u,v\rangle-\langle\pi_{-1}u,v\rangle-\langle\pi_{1}u,\pi_{0}v\rangle+\langle\pi_{1}u,\pi_{1}v\rangle-\\
 & \quad-\langle\pi_{1}u,\pi_{-1}v\rangle+\langle\pi_{-1}u,\pi_{0}v\rangle+\langle\pi_{-1}u,\pi_{1}v\rangle-\langle\pi_{-1}u,\pi_{-1}v\rangle\\
 & =2\langle\pi_{1}u,\pi_{1}v\rangle-2\langle\pi_{-1}u,\pi_{-1}v\rangle\\
 & =\langle\pi_{1}u,\pi_{1}v\rangle_{A}-\langle\pi_{-1}u,\pi_{-1}v\rangle_{A},
\end{align*}
which proves that $\left((\pi_{1},\pi_{-1}),\ker(1-A),\ker(1+A)\right)$
is a boundary system for $A$. 
\end{proof}
We now recall the main result of \cite{Picard_Trostorff_2023}.
\begin{thm}[{\cite[Theorem 3.1]{Picard_Trostorff_2023}}]
\label{thm:Rainer_main} Let $A_{0}\colon\dom(A_{0})\subseteq H\to H$
be a skew-symmetric closed linear operator and set $A\coloneqq-A_{0}^{\ast}$
and consider the boundary system $((\pi_{1},\pi_{-1}),\ker(1-A),\ker(1+A))$
from \prettyref{lem:standard_bd_sys}. Moreover, let $B\subseteq A$
be any (possibly nonlinear) restriction. Then $B$ is m-accretive
if and only if there exists a mapping 
\[
h\colon\ker(1-A)\to\ker(1+A)
\]
with 
\[
\forall x,y\in\ker(1-A):\,\|h(x)-h(y)\|_{A}\leq\|x-y\|_{A}
\]
such that 
\[
\dom(B)=\{u\in\dom(A)\,;\,h(\pi_{1}u)=\pi_{-1}u\}.
\]
\end{thm}

\begin{rem}
In \cite[Theorem 3.1]{Picard_Trostorff_2023} the Lipschitz estimate
for $h$ is stated with respect to the norm in $H$. This however
is equivalent to the above estimate, since on $\ker(1-A)$ and $\ker(1+A)$
both norms just differ by the constant $\sqrt{2}.$ 
\end{rem}

\section{Main result}

Throughout, let $A_{0}\colon\dom(A_{0})\subseteq H\to H$ be a skew-symmetric
closed linear operator on a Hilbert space $H$ and $A\coloneqq-A_{0}^{\ast}.$
We will prove the following theorem, which generalises \prettyref{thm:Rainer_main}
to arbitrary boundary systems. 
\begin{thm}
\label{thm:main}Let $(F,G_{1},G_{2})$ be a boundary system for $A$.
Then $B\subseteq A$ is m-accretive if and only if there exists a
mapping 
\[
g\colon G_{1}\to G_{2}
\]
with
\[
\forall a,b\in G_{1}:\,\|g(a)-g(b)\|_{G_{2}}\leq\|a-b\|_{G_{1}}
\]
such that 
\[
\dom(B)=\{u\in\dom(A)\,;\,g(F_{1}u)=F_{2}u\}.
\]
\end{thm}

We will prove this theorem by using \prettyref{thm:Rainer_main} and
Kirszbraun's surprising theorem on Lipschitz-continuous extensions
(see e.g. \cite[1.31. Theorem]{Schwartz1969}). 
\begin{thm}[Kirszbraun, \cite{Kirszbraun_1934}]
 \label{thm:Kirszbraun}Let $H_{0},H_{1}$ be two Hilbert spaces
and $f\colon D\subseteq H_{0}\to H_{1}$ be a Lipschitz-continuous
mapping defined on some subset $D\subseteq H_{0}.$ Then there exists
a Lipschitz-continuous mapping $g\colon H_{0}\to H_{1}$ with $g|_{D}=f$
and $|g|_{\mathrm{Lip}}=|f|_{\mathrm{Lip}}$.
\end{thm}

We start with a simple observation.
\begin{lem}
\label{lem:accretive}Let $(F,G_{1},G_{2})$ be a boundary system
for $A$ and let $B\subseteq A$. Then $B$ is accretive, if and only
if there exists a contraction $g\colon G_{1}\to G_{2}$ with 
\[
\dom(B)\subseteq\{u\in\dom(A)\,;\,g(F_{1}u)=F_{2}u\}.
\]
\end{lem}

\begin{proof}
If $B$ is accretive, we consider 
\[
f\coloneqq\{(F_{1}u,F_{2}u)\in G_{1}\times G_{2}\,;\,u\in\dom(B)\}.
\]
If $u,v\in\dom(B)$ then 
\[
\|F_{1}(u-v)\|_{G_{1}}^{2}-\|F_{2}(u-v)\|_{G_{2}}^{2}=2\langle A(u-v),u-v\rangle=2\langle B(u)-B(v),u-v\rangle\geq0
\]
and thus, 
\[
\|F_{2}(u)-F_{2}(v)\|_{G_{2}}^{2}\leq\|F_{1}(u)-F_{1}(v)\|_{G_{1}}^{2}.
\]
This shows, that $f$ is a contractive function defined on $D\coloneqq F_{1}[\dom(B)]$.
If we now apply \prettyref{thm:Kirszbraun}, we can extend $f$ to
a contraction $g\colon G_{1}\to G_{2}$ and clearly 
\[
\dom(B)\subseteq\{u\in\dom(A)\,;\,g(F_{1}u)=F_{2}u\}.
\]
On the other hand, if $\dom(B)\subseteq\{u\in\dom(A)\,;\,g(F_{1}u)=F_{2}u\}$
we estimate for each $u,v\in\dom(B)$
\begin{align*}
\langle B(u)-B(v),u-v\rangle & =\langle A(u-v),u-v\rangle\\
 & =\frac{1}{2}\left(\|F_{1}(u-v)\|_{G_{1}}^{2}-\|F_{2}(u-v)\|_{G_{2}}^{2}\right)\\
 & =\frac{1}{2}\left(\|F_{1}(u)-F_{1}(v)\|_{G_{1}}^{2}-\|g(F_{1}u)-g(F_{1}v)\|_{G_{2}}^{2}\right)\geq0.\tag*{\qedhere}
\end{align*}
\end{proof}
\begin{cor}
\label{cor:m-accretive_yields_contraction}Let $(F,G_{1},G_{2})$
be a boundary system for $A$ and let $B\subseteq A$ be m-accretive.
Then there exists a contraction $g\colon G_{1}\to G_{2}$ with 
\[
\dom(B)=\{u\in\dom(A)\,;\,g(F_{1}u)=F_{2}u\}.
\]
\end{cor}

\begin{proof}
By \prettyref{lem:accretive} we find a contraction $g\colon G_{1}\to G_{2}$
with 
\[
\dom(B)\subseteq\{u\in\dom(A)\,;\,g(F_{1}u)=F_{2}u\}.
\]
If we now set $\tilde{B}\subseteq A$ with $\dom(\tilde{B})=\{u\in\dom(A)\,;\,g(F_{1}u)=F_{2}u\},$
then $\tilde{B}$ is accretive again by \prettyref{lem:accretive}
with $B\subseteq\tilde{B}.$ By maximality, we infer that $B=\tilde{B}$
(cp. \prettyref{rem:m-accretive} (a)). 
\end{proof}
\begin{rem}
The latter corollary can also be proved without the help of \prettyref{lem:accretive}
and hence, without Kirszbraun's Theorem. Indeed, one can define the
function $g$ on $F_{1}[\dom(B)]$ as in the proof of \prettyref{lem:accretive}.
Since $B$ is m-accretive, $\overline{\dom}(B)$ is convex (see e.g.
\cite[Proposition 17.2.2]{STW2022}), and hence, so is $\overline{F_{1}[\dom(B)]}=\overline{F_{1}[\overline{\dom}(B)]}.$
Indeed, this follows from the continuity and linearity of $F_{1}$
(see \prettyref{lem:bd_systems_easy}). Now, $g$ can clearly be extended
to a contraction on $\overline{F_{1}[\dom(B)]}$ and if $P$ denotes
the projection onto $\overline{F_{1}[\dom(B)]},$ we can extend $g$
to the whole Hilbert space by taking $g\circ P$ (note that $P$ is
also a contraction). 
\end{rem}

We can now prove our main result.
\begin{proof}[Proof of \prettyref{thm:main}]
 If $B\subseteq A$ is m-accretive the assertion is shown in \prettyref{cor:m-accretive_yields_contraction}.
Conversely, assume that 
\[
\dom(B)=\{u\in\dom(A)\,;\,\,g(F_{1}u)=F_{2}u\}
\]
for some contraction $g\colon G_{1}\to G_{2}.$ Then by \prettyref{lem:accretive},
$B$ is accretive. Recall that
\[
\left((\pi_{1},\pi_{-1}),\ker(1-A),\ker(1+A)\right)
\]
 is a boundary system for $A$ by \prettyref{lem:standard_bd_sys}.
Hence, by \prettyref{lem:accretive} we find a Lipschitz-continuous
mapping $h\colon\ker(1-A)\to\ker(1+A)$ with 
\[
\dom(B)\subseteq\{u\in\dom(A)\,;\,h(\pi_{1}u)=\pi_{-1}u\}.
\]
We define $\tilde{B}\subseteq A$ with $\dom(\tilde{B})\coloneqq\{u\in\dom(A)\,;\,h(\pi_{1}u)=\pi_{-1}u\}.$
Then by \prettyref{thm:Rainer_main}, $\tilde{B}$ is m-accretive
with $B\subseteq\tilde{B}.$ Moreover, by \prettyref{cor:m-accretive_yields_contraction}
we find a contraction $\tilde{g}\colon G_{1}\to G_{2}$ with 
\[
\dom(\tilde{B})=\{u\in\dom(A)\,;\,\tilde{g}(F_{1}u)=F_{2}u\}.
\]
 To complete the proof, we show that $\tilde{g}=g.$ Let $a\in G_{1}$
and consider the point $(a,g(a))\in G_{1}\times G_{2}.$ Since $F$
is onto, there exists $u\in\dom(A)$ with $F_{1}u=a$ and $F_{2}u=g(a).$
This yields $u\in\dom(B)\subseteq\dom(\tilde{B})$ and hence, $\tilde{g}(a)=\tilde{g}(F_{1}u)=F_{2}u=g(a);$
that is, $\tilde{g}=g.$ 
\end{proof}
\begin{prop}
\label{prop:linear} Let $(F,G_{1},G_{2})$ be a boundary system for
$A$. Let $g\colon G_{1}\to G_{2}$ be a contraction and set $B\subseteq A$
with 
\[
\dom(B)\coloneqq\{u\in\dom(A)\,;\,g(F_{1}u)=F_{2}u\};
\]
that is, $B$ is an m-accretive restriction of $A$. Then

\begin{enumerate}[(a)]

\item $A_{0}\subseteq B$ if and only if $g(0)=0$,

\item $B$ is linear if and inly if $g$ is linear. In particular,
if $B$ is linear then $A_{0}\subseteq B$. 

\end{enumerate}
\end{prop}

\begin{proof}
(a) We have $A_{0}\subseteq B$ if and only if $\dom(A_{0})\subseteq\dom(B).$
Since $\dom(A_{0})=\ker F$ by \prettyref{lem:bd_systems_easy}, the
latter is equivalent to $g(0)=0.$ \\
(b) Assume that $B$ is linear and let $a,b\in G_{1}$ as well as
$\lambda\in\R.$ Since $F$ is onto, we find $u,v\in\dom(A)$ such
that $F(u)=(a,g(a)),F(v)=(b,g(b)).$ Consequently, $u,v\in\dom(B)$
and clearly, 
\[
F(\lambda u+v)=(\lambda a+b,\lambda g(a)+g(b)).
\]
Since $B$ is linear, $\lambda u+v\in\dom(B)$ and thus, $g(\lambda a+b)=g(F_{1}(\lambda u+v))=F_{2}(\lambda u+v)=\lambda g(a)+g(b)$,
showing that $g$ is linear. If on the other hand $g$ is linear,
the linearity of $B$ follows trivially. The additional assertion
is a direct consequence of (a).
\end{proof}
\begin{cor}
\label{cor:extension_vs_restriction} Let $B\supseteq A_{0}$ be m-accretive
and linear. Then $B\subseteq A.$
\end{cor}

\begin{proof}
Let $(F,G_{1},G_{2})$ be an arbitrary boundary system for $A$ (such
a system always exists by \prettyref{lem:standard_bd_sys}). Since
$B\supseteq A_{0}$ is m-accretive, so is $B^{\ast}\subseteq-A_{0}^{\ast}=A$.
Indeed, since $B$ is m-accretive, we get $\R_{<0}\subseteq\rho(B)$
with $\|(\mu+B)^{-1}\|\leq\frac{1}{\mu}$ for each $\mu>0$ by \prettyref{lem:onto_m_accretive}.
Since this clearly carries over to $B^{\ast},$ we infer that $B^{\ast}$
is m-accretive as well again by \prettyref{lem:onto_m_accretive}.
Hence, by \prettyref{thm:main} and \prettyref{prop:linear} (b),
we infer that $A_{0}\subseteq B^{\ast}$, which in turn implies $B\subseteq A.$ 
\end{proof}

\section{Application to Port-Hamiltonian systems}

We apply our previous findings to port-Hamiltonian systems of arbitrary
order. We begin to recall the setting. Throughout let $a,b\in\R$
with $a<b$. Moreover, let $d,n\in\N$ and $P_{0},\ldots,P_{n}\in\R^{d\times d}$
with $P_{j}^{\top}=(-1)^{j+1}P_{j}$ and $\det P_{n}\ne0.$ Finally,
let $\mathcal{H}\in L_{\infty}(a,b;\R^{d\times d})$ be such that
$\mathcal{H}(x)$ is symmetric for almost every $x\in[a,b]$ and there
exists $c>0$ such that 
\[
\langle\mathcal{H}(x)v,v\rangle_{\R^{d}}\geq c\|v\|_{\R^{d}}^{2}\quad(v\in\R^{d},x\in[a,b]\text{ a.e.}).
\]
We consider the following operator 
\[
A\colon\dom(A)\subseteq H\to H,\quad Au=\sum_{k=0}^{n}P_{k}\partial^{k}(\mathcal{H}u),
\]
where $H=L_{2}(a,b)^{d}$ equipped with the inner product 
\[
\langle u,v\rangle_{H}\coloneqq\int_{a}^{b}\langle\mathcal{H}(x)u(x),v(x)\rangle\d x\quad(u,v\in L_{2}(a,b)^{d})
\]
and 
\[
\dom(A)\coloneqq\{u\in H\,;\,\mathcal{H}u\in H^{n}(a,b)^{d}\}.
\]
Note that the norm on $H$ is equivalent to the standard $L_{2}$-norm.
We will characterise all m-accretive restrictions of $A$ by using
the theory of boundary systems. First, we show that it suffices to
treat the case $\mathcal{H}(x)=1_{d\times d}$ for $x\in[a,b]$, see
also \cite[Lemma 7.2.3]{JacobZwart2012} or \cite[Proposition 3.1]{PTW2023}. 
\begin{lem}
\label{lem:congruence}Consider the operator 
\begin{align*}
S\colon H & \to L_{2}(a,b)^{n},\\
u & \mapsto\mathcal{H}u.
\end{align*}
Then the following statements hold:

\begin{enumerate}[(a)]

\item $S$ is a bijection and $S^{\ast}\colon L_{2}(a,b)^{n}\to H$
is given by $S^{\ast}v=v$ for $v\in L_{2}(a,b)^{n}.$ 

\item A restriction $B\subseteq A$ is m-accretive on $H$ if and
only if the operator $\tilde{B}\subseteq\sum_{k=0}^{n}P_{k}\partial^{k}$
with 
\[
\dom(\tilde{B})=\{u\in H^{n}(a,b)\,;\,\mathcal{H}^{-1}u\in\dom(B)\}
\]
 is m-accretive on $L_{2}(a,b)^{n}.$

\end{enumerate}
\end{lem}

\begin{proof}
(a) From the strict positive definiteness of $\mathcal{H}$, we infer
that $\mathcal{H},\mathcal{H}^{-1}\in L_{\infty}(a,b;\R^{d\times d}).$
Hence, $S$ is clearly a bijection. Moreover, for $u\in H,v\in L_{2}(a,b)^{n}$
we compute 
\[
\langle Su,v\rangle_{L_{2}}=\langle\mathcal{H}u,v\rangle_{L_{2}}=\langle u,v\rangle_{H},
\]
which shows $S^{\ast}v=v.$ 

(b) Assume that $\tilde{B}$ is m-accretive. By (a) we have $S^{\ast}\tilde{B}S=B.$
Hence, $B$ is accretive. Moreover, note that $\left(S^{\ast}\right)^{-1}S^{-1}$
is an accretive selfadjoint and bijective operator. Then there exists
$\kappa>0$ such that 
\[
\langle\left(S^{\ast}\right)^{-1}S^{-1}u,u\rangle_{L_{2}}\geq\kappa\|u\|_{L_{2}}^{2}\quad(u\in L_{2}(a,b)^{n});
\]
that is, $T\coloneqq\left(S^{\ast}\right)^{-1}S^{-1}-\kappa$ is accretive.
Consider now the mapping $T+\tilde{B},$ which is clearly accretive.
Moreover, for $\lambda>\|T\|$ and $f\in L_{2}(a,b)^{n},$ we infer
that 
\[
u\mapsto(\lambda+\tilde{B})^{-1}\left(f-Tu\right)
\]
is a strict contraction (recall that $|(\lambda+\tilde{B})^{-1}|_{\mathrm{Lip}}\leq\frac{1}{\lambda}$
by \prettyref{lem:onto_m_accretive}) and hence, possesses a fixed
point $u\in L_{2}(a,b)^{n}.$ For this fixed point we have 
\[
\lambda u+Tu+\tilde{B}(u)=f
\]
and thus, $\lambda+T+\tilde{B}$ is onto. By \prettyref{lem:onto_m_accretive}
we infer that also $\left(S^{\ast}\right)^{-1}S^{-1}+\tilde{B}=\kappa+T+\tilde{B}$
is onto. Hence, for $g\in H$ we find $u\in L_{2}(a,b)^{n}$ such
that 
\[
\left(S^{\ast}\right)^{-1}S^{-1}u+\tilde{B}(u)=\left(S^{\ast}\right)^{-1}g
\]
and hence $v\coloneqq S^{-1}u$ satisfies 
\[
g=v+S^{\ast}\tilde{B}(Sv)=v+B(v),
\]
which shows that $B$ is m-accretive. The other implication follows
by arguing the same lines interchanging the roles of $B$ and $\tilde{B}$.
\end{proof}
So in the following, let $\mathcal{H}(x)=1_{d\times d}$ for $x\in[a,b]$
and $H=L_{2}(a,b)^{d}$ be equipped with the standard inner product.
We show that the operator 
\begin{equation}
A=\sum_{k=0}^{n}P_{k}\partial^{k}:H^{n}(a,b)^{d}\subseteq L_{2}(a,b)^{d}\to L_{2}(a,b)^{d}\label{eq:A}
\end{equation}
fits in the abstract framework of the previous sections. For doing
so, we define the operator 
\[
A_{0}\coloneqq\sum_{k=0}^{n}P_{k}\partial_{0}^{k}
\]
with domain $\dom(A_{0})\coloneqq H_{0}^{n}(a,b)^{d}=\{u\in H^{n}(a,b)^{d}\,;\,\forall k\in\{0,\ldots,n-1\}:u^{(k)}(a)=u^{(k)}(b)=0\}.$
We will show that this operator is indeed closed and skew-symmetric
and that its negative adjoint is precisely the operator $A$. Note
that since $\rho(\partial_{0})=\emptyset,$ it is not a priori clear
that $A_{0}$ is closed. We define 
\begin{align*}
I\colon L_{2}(a,b)^{d} & \to L_{2}(a,b)^{d}\\
f & \mapsto\left(t\mapsto\int_{a}^{t}f(s)\d s\right).
\end{align*}
Then it is easy to see that $I$ is bounded and that 
\[
I\partial_{0}f=f\quad(f\in H_{0}^{1}(a,b)^{d}).
\]
Since $\dom(A_{0})=H_{0}^{n}(a,b)^{d},$ we infer that 
\[
A_{0}=\sum_{k=0}^{n}P_{k}\partial_{0}^{k}=\sum_{k=0}^{n}P_{k}I^{n-k}\partial_{0}^{n}.
\]
We begin to show that $\partial_{0}^{n}$ is closed.
\begin{lem}
\label{lem:derivative_closed}For each $n\in\N$ the operator $\partial_{0}^{n}$
is closed and $C_{c}^{\infty}(a,b)^{d}$ is a core for $\partial_{0}^{n}.$
\end{lem}

\begin{proof}
The proof if done by induction. For $n=1$ the claim follows by definition
of $\partial_{0}.$ Let now $\partial_{0}^{n}$ be closed for some
$n\in\N$ and let $(f_{k})_{k\in\N}$ be a sequence in $\dom(\partial_{0}^{n+1})$
such that $f_{k}\to f$ and $\partial_{0}^{n+1}f_{k}\to g$ for some
$f,g\in L_{2}(a,b)^{d}.$ By the boundedness of $I$ we infer that
\[
\partial_{0}^{n}f_{k}=I\partial_{0}^{n+1}f_{k}\to Ig
\]
and hence, by induction hypothesis we infer $f\in\dom(\partial_{0}^{n})$
and $\partial_{0}^{n}f=Ig.$ Moreover, $\partial_{0}(\partial_{0}^{n}f_{k})\to g$
and $\partial_{0}^{n}f_{k}\to Ig=\partial_{0}^{n}f$, and hence, by
the closedness of $\partial_{0}$, we derive $f\in\dom(\partial_{0}^{n+1})$
with $\partial_{0}^{n+1}f=g,$ which shows that $\partial_{0}^{n+1}$
is closed. \\
That $C_{c}^{\infty}(a,b)^{d}$ is a core for $\partial_{0}$ follows
again by definition. We assume that $C_{c}^{\infty}(a,b)^{d}$ is
a core for $\partial_{0}^{n}$ for some $n\in\N$. Let $f\in\dom(\partial_{0}^{n+1}).$
Then $\partial_{0}f\in\dom(\partial_{0}^{n})$ and we find a sequence
$(\psi_{k})_{k}$ in $C_{c}^{\infty}(a,b)^{d}$ such that $\psi_{k}\to\partial_{0}f$
and $\psi_{k}^{(n)}\to\partial_{0}^{n+1}f$. Fix now a function $\varphi\in C^{\infty}(a,b)$
such that $\varphi=0$ near $a$ and $\varphi=1$ near $b$. We define
\[
\varphi_{k}\coloneqq I\psi_{k}-\varphi\int_{a}^{b}\psi_{k}(s)\d s\quad(k\in\N).
\]
Then $\varphi_{k}\in C_{c}^{\infty}(a,b)^{d}$ for each $k\in\N$
and 
\[
\varphi_{k}\to I\partial_{0}f-\varphi\int_{a}^{b}\partial_{0}f(s)\d s=f,
\]
where we have used that $\int_{a}^{b}\partial_{0}f(s)\d s=0,$ since
$f\in H_{0}^{1}(a,b)^{d}.$ Analogously 
\[
\varphi_{k}^{(n+1)}=\psi_{k}^{(n)}-\varphi^{(n+1)}\int_{a}^{b}\psi_{k}(s)\d s\to\partial_{0}^{n+1}f
\]
and hence, $C_{c}^{\infty}(a,b)^{d}$ is a core for $\partial_{0}^{n+1}.$ 
\end{proof}
\begin{prop}
The operator $A_{0}$ is closed and $C_{c}^{\infty}(a,b)^{d}$ is
a core for $A_{0}.$
\end{prop}

\begin{proof}
Since $\partial_{0}^{n}$ is closed by \prettyref{lem:derivative_closed}
it suffices to show that $\sum_{k=0}^{n}P_{k}I^{n-k}$ is boundedly
invertible to obtain the closedness of $A_{0}$. For doing so, we
equip $L_{2}(a,b)^{d}$ with an equivalent norm given by 
\[
\|f\|_{\rho}\coloneqq\left(\int_{a}^{b}\|f(t)\|^{2}\e^{-2\rho t}\d t\right)^{\frac{1}{2}}\quad(f\in L_{2}(a,b)^{d})
\]
for some $\rho\in\R_{>0}$. Then 
\begin{align*}
\|If\|_{\rho}^{2} & =\int_{a}^{b}\|If(t)\|^{2}\e^{-2\rho t}\d t\\
 & =\int_{a}^{b}\|\int_{a}^{t}f(s)\d s\|^{2}\e^{-2\rho t}\d t\\
 & \leq\int_{a}^{b}\int_{a}^{t}\|f(s)\|^{2}\e^{-2\rho s}\d s\frac{1}{2\rho}\left(\e^{2\rho t}-\e^{2\rho a}\right)\e^{-2\rho t}\d t\\
 & \leq\frac{(b-a)}{2\rho}\|f\|_{\rho}^{2}
\end{align*}
for each $f\in L_{2}(a,b)^{d}$ and $\rho>0.$ Since, 
\[
\sum_{k=0}^{n}P_{k}I^{n-k}=P_{n}+\sum_{k=0}^{n-1}P_{k}I^{n-k}=P_{n}\left(1+\sum_{k=0}^{n-1}P_{n}^{-1}P_{k}I^{n-k}\right)
\]
and $\|\sum_{k=0}^{n-1}P_{n}^{-1}P_{k}I^{n-k}\|<1$ if we choose $\rho$
large enough, we infer that $\sum_{k=0}^{n}P_{k}I^{n-k}$ is indeed
boundedly invertible. Thus, $A_{0}$ is closed. Moreover, since $C_{c}^{\infty}(a,b)^{d}$
is a core for $\partial_{0}^{n}$ by \prettyref{lem:derivative_closed}
and $\sum_{k=0}^{n}P_{k}I^{n-k}$ is bounded, $C_{c}^{\infty}(a,b)^{d}$
is also a core for $A_{0}.$ 
\end{proof}
We now compute the adjoint of $A_{0}.$ We start with the following
prerequisite.
\begin{lem}
For each $n\in\N$ we have $(\partial_{0}^{n})^{\ast}=(-1)^{n}\partial^{n}.$
\end{lem}

\begin{proof}
Since in general $B^{\ast}A^{\ast}\subseteq(AB)^{\ast}$ for two operators
$A$ and $B$, we infer that $(-1)^{n}\partial^{n}\subseteq(\partial_{0}^{n})^{\ast}.$
Let now $u\in\dom\left((\partial_{0}^{n})^{\ast}\right).$ We first
show that $u\in\dom(\partial).$ Let $\zeta\in C^{\infty}(a,b)$ with
$\zeta=1$ near $b$ and $\zeta=0$ near $a$ and define the bounded
operator
\[
Jf\coloneqq If-\zeta\int_{a}^{b}f\quad(f\in L_{2}(a,b)^{d}).
\]
Then $J$ leaves $C_{c}^{\infty}(a,b)^{d}$ invariant. Let now $\psi\in C_{c}^{\infty}(a,b)^{d}$
and set $\varphi\coloneqq J^{n-1}\psi\in C_{c}^{\infty}(a,b)^{d}$.
Then inductively we compute 
\[
\varphi^{(k)}=J^{n-k-1}\psi-\sum_{j=1}^{k}\zeta^{(j)}\int_{a}^{b}J^{n-k+j-2}\psi
\]
for each $k\in\{1,\ldots,n-1\}$. In particular 
\[
\varphi^{(n-1)}=\psi-\sum_{j=1}^{n-1}\zeta^{(j)}\int_{a}^{b}J^{j-1}\psi
\]
and hence, 
\[
\varphi^{(n)}=\psi'-\sum_{j=1}^{n-1}\zeta^{(j+1)}\int_{a}^{b}J^{j-1}\psi.
\]
Thus, 
\begin{align*}
\langle u,\psi'\rangle & =\langle u,\varphi^{(n)}\rangle+\sum_{j=1}^{n-1}\langle u,\zeta^{(j+1)}\int_{a}^{b}J^{j-1}\psi\rangle\\
 & =\langle(\partial_{0}^{n})^{\ast}u,\varphi\rangle+\sum_{j=1}^{n-1}\langle\int_{a}^{b}\zeta^{(j+1)}u,J^{j-1}\psi\rangle\\
 & =\langle\left(J^{n-1}\right)^{\ast}(\partial_{0}^{n})^{\ast}u,\psi\rangle+\langle\sum_{j=1}^{n-1}\left(J^{j-1}\right)^{\ast}\int_{a}^{b}\zeta^{(j+1)}u,\psi\rangle,
\end{align*}
which shows $u\in\dom(\partial)$ with 
\[
\partial u=-\left(J^{n-1}\right)^{\ast}(\partial_{0}^{n})^{\ast}u-\sum_{j=1}^{n-1}\left(J^{j-1}\right)^{\ast}\int_{a}^{b}\zeta^{(j+1)}u.
\]
Since $J^{\ast}$ maps $L_{2}(a,b)^{d}$ to $H^{1}(a,b)^{d}$ and
constants to $H^{\infty}(a,b)^{d}$, we derive that $u\in\dom(\partial^{n}),$
which completes the proof. 
\end{proof}
\begin{prop}
\label{prop:port_H}The adjoint of $A_{0}$ is given by $-A$.
\end{prop}

\begin{proof}
Recall that $A_{0}=\sum_{k=0}^{n}P_{k}I^{n-k}\partial_{0}^{n}$ and
hence, 
\begin{align*}
A_{0}^{\ast} & =\left(\partial_{0}^{n}\right)^{\ast}\left(\sum_{k=0}^{n}P_{k}I^{n-k}\right)^{\ast}\\
 & =(-1)^{n}\partial^{n}\sum_{k=0}^{n}\left(I^{n-k}\right)^{\ast}P_{k}^{\top}\\
 & =(-1)^{n}\partial^{n}P_{n}^{\top}+(-1)^{n}\partial^{n}\sum_{k=0}^{n-1}\left(I^{n-k}\right)^{\ast}P_{k}^{\top}.
\end{align*}
Thus, in particular, $\dom(A_{0}^{\ast})\subseteq\dom(\partial^{n})$
since $P_{n}$ is invertible. Moreover, an easy computation reveals
\[
I^{\ast}f(t)=\int_{t}^{b}f(s)\d s\quad(t\in[a,b],f\in L_{2}(a,b)^{d})
\]
and thus, $\partial I^{\ast}f=-f$ for each $f\in L_{2}(a,b)^{d}.$
The latter gives 
\begin{align*}
(-1)^{n}\partial^{n}\sum_{k=0}^{n-1}\left(I^{n-k}\right)^{\ast}P_{k}^{\top} & =(-1)^{n}\sum_{k=0}^{n-1}(-1)^{n-k}\partial^{k}P_{k}^{\top}\\
 & =-\sum_{k=0}^{n-1}\partial^{k}P_{k}\\
 & =-\sum_{k=0}^{n-1}P_{k}\partial^{k},
\end{align*}
on $\dom(A_{0}^{\ast})$, where we have used that $\dom(A_{0}^{\ast})\subseteq\dom(\partial^{n})$
in the last equality. Altogether, we infer 
\[
A_{0}^{\ast}=(-1)^{n}\partial^{n}P_{n}^{\top}-\sum_{k=0}^{n-1}P_{k}\partial^{k}=-A.\tag*{\qedhere}
\]
\end{proof}
We define the matrix 
\begin{equation}
Q\coloneqq\left(\begin{array}{ccccc}
P_{1} & P_{2} & \cdots & \cdots & P_{n}\\
-P_{2} & -P_{3} & \cdots & -P_{n} & 0\\
\vdots &  & \iddots\\
\vdots & \iddots\\
(-1)^{n+1}P_{n} & 0 &  &  & 0
\end{array}\right)\in\R^{nd\times nd}.\label{eq:Q}
\end{equation}
Then $Q$ is symmetric and invertible. Hence, we can orthogonally
decompose $\R^{nd}$ into the spaces 
\begin{align*}
E_{+} & \coloneqq\lin\{x\in\R^{nd}\,;\,\exists\lambda>0:Qx=\lambda x\}\\
E_{-} & \coloneqq\lin\{x\in\R^{nd}\,;\,\exists\lambda<0:Qx=\lambda x\},
\end{align*}
such that 
\[
\R^{nd}=E_{+}\oplus E_{-}.
\]
We denote the canonical embeddings of $E_{+}$ and $E_{-}$ in $\R^{nd}$
by $\iota_{+}\colon E_{+}\to\R^{nd}$ and $\iota_{-}\colon E_{-}\to\R^{nd}$
respectively. Then the operators $\iota_{+}^{\ast}Q\iota_{+}$ and
$\iota_{-}^{\ast}(-Q)\iota_{-}$ define positive selfadjoint operators
on $E_{+}$ and $E_{-}$ respectively, and we set 
\begin{equation}
Q_{+}\coloneqq\iota_{+}\left(\iota_{+}^{\ast}Q\iota_{+}\right)^{\frac{1}{2}}\iota_{+}^{\ast},\,Q_{-}\coloneqq\iota_{-}\left(\iota_{-}^{\ast}(-Q)\iota_{-}\right)^{\frac{1}{2}}\iota_{-}^{\ast}.\label{eq:Q_+}
\end{equation}
Note that 
\[
Q=Q_{+}^{2}-Q_{-}^{2}.
\]
Moreover, we define the operator 
\begin{equation}
\tr_{t}\colon H^{n}(a,b)^{d}\to\R^{nd},\quad u\mapsto\left(\begin{array}{c}
u(t)\\
\vdots\\
u^{(n-1)}(t)
\end{array}\right)\label{eq:trace}
\end{equation}
for $t\in[a,b].$ Note that this operator is well-defined by the Sobolev
embedding theorem. 
\begin{lem}
\label{lem:trace_onto}The operator 
\[
\left(\begin{array}{c}
\tr_{b}\\
\tr_{a}
\end{array}\right)\colon H^{n}(a,b)^{d}\to\R^{nd}\times\R^{nd}
\]
is onto.
\end{lem}

\begin{proof}
This is a consequence of the so-called Hermite interpolation. We show
that the mapping 
\[
\left(\begin{array}{c}
\tr_{b}\\
\tr_{a}
\end{array}\right)\colon P_{2n-1}(a,b)^{d}\to\R^{nd}\times\R^{nd}
\]
is onto, where $P_{2n-1}(a,b)$ denotes the space of polynomials on
$(a,b)$ up to degree $2n-1.$ Since both spaces are $2nd$-dimensional,
it suffices to check that the above mapping is one-to-one. Moreover,
it suffices to treat the case $d=1,$ since in the general case one
can argue coordinate-wise. So, let $p\in P_{2n-1}(a,b)$ such that
\[
\left(\begin{array}{c}
\tr_{b}\\
\tr_{a}
\end{array}\right)p=0,
\]
that is 
\[
p(a)=\ldots=p^{(n-1)}(a)=p(b)=\ldots=p^{(n-1)}(b)=0.
\]
We represent $p$ by 
\[
p(x)=\sum_{j=0}^{n-1}c_{j}(x-a)^{j}+(x-a)^{n}\sum_{k=0}^{n-1}d_{k}(x-b)^{k}\quad(x\in[a,b]),
\]
for suitable $c_{0},\ldots,c_{n-1},d_{0},\ldots,d_{n-1}\in\R.$ Since
the derivatives up to order $n-1$ of the second part $(x-a)^{n}\sum_{k=0}^{n-1}d_{k}(x-b)^{k}$
vanish in $a$, we infer that 
\[
0=p^{(j)}(a)=j!c_{j}\quad(j\in\{0,\ldots,n-1\}),
\]
and thus, $c_{0}=\ldots=c_{n-1}=0.$ Hence, 
\[
p(x)(x-a)^{-n}=\sum_{k=0}^{n-1}d_{k}(x-b)^{k}\quad(x\in[a,b]).
\]
Differentiating both sides $k\leq n-1$ times and evaluate them at
$b$, we obtain 
\[
0=k!d_{k}
\]
and thus, $d_{0}=\ldots=d_{n-1}=0,$ which shows $p=0.$ 
\end{proof}
\begin{prop}
\label{prop:bd_system_pH}Let $A$ be as in \prettyref{eq:A}. Then
$\left(F,\R^{nd},\R^{nd}\right)$ with 
\[
Fu\coloneqq\left(\begin{array}{c}
Q_{+}\tr_{b}u+Q_{-}\tr_{a}u\\
Q_{-}\tr_{b}u+Q_{+}\tr_{a}u
\end{array}\right)\quad(u\in H^{n}(a,b)^{d})
\]
defines a boundary system for $A$.
\end{prop}

\begin{proof}
Let $u,v\in\dom(A)=H^{n}(a,b)^{d}.$ Then we compute 
\begin{align*}
\langle Au,v\rangle_{L_{2}} & =\sum_{k=0}^{n}\langle P_{k}u^{(k)},v\rangle_{L_{2}}\\
 & =\langle P_{0}u,v\rangle_{L_{2}}+\sum_{k=1}^{n}\langle P_{k}u^{(k)},v\rangle_{L_{2}}.
\end{align*}
For $k\in\{1,\ldots,n\}$ we further compute 
\begin{align*}
\langle u^{(k)},v\rangle_{L_{2}} & =\langle u^{(k-1)}(b),v(b)\rangle-\langle u^{(k-1)}(a),v(a)\rangle-\langle u^{(k-1)},v'\rangle_{L_{2}}\\
 & \vdots\\
 & =\sum_{j=1}^{k}(-1)^{j+1}\left(\langle u^{(k-j)}(b),v^{(j-1)}(b)\rangle-\langle u^{(k-j)}(a),v^{(j-1)}(a)\rangle\right)+(-1)^{k}\langle u,v^{(k)}\rangle_{L_{2}}
\end{align*}
and hence, 
\begin{align*}
 & \sum_{k=1}^{n}\langle P_{k}u^{(k)},v\rangle_{L_{2}}\\
 & =\sum_{k=1}^{n}\left(\sum_{j=1}^{k}(-1)^{j+1}\left(\langle P_{k}u^{(k-j)}(b),v^{(j-1)}(b)\rangle-\langle P_{k}u^{(k-j)}(a),v^{(j-1)}(a)\rangle\right)+(-1)^{k}\langle P_{k}u,v^{(k)}\rangle_{L_{2}}\right)\\
 & =\sum_{j=1}^{n}\left(\langle\sum_{k=j}^{n}(-1)^{j+1}P_{k}u^{(k-j)}(b),v^{(j-1)}(b)\rangle-\langle\sum_{k=j}^{n}(-1)^{j+1}P_{k}u^{(k-j)}(a),v^{(j-1)}(a)\rangle\right)-\langle u,\sum_{k=1}^{n}P_{k}v^{(k)}\rangle_{L_{2}}\\
 & =\langle Q\tr_{b}u,\tr_{b}v\rangle-\langle Q\tr_{a}u,\tr_{a}v\rangle-\langle u,\sum_{k=1}^{n}P_{k}v^{(k)}\rangle_{L_{2}}.
\end{align*}
Summarising, we have 
\begin{align*}
\langle Au,v\rangle & =\langle Q\tr_{b}u,\tr_{b}v\rangle-\langle Q\tr_{a}u,\tr_{a}v\rangle-\langle u,\sum_{k=1}^{n}P_{k}v^{(k)}\rangle_{L_{2}}+\langle P_{0}u,v\rangle_{L_{2}}\\
 & =\langle Q\tr_{b}u,\tr_{b}v\rangle-\langle Q\tr_{a}u,\tr_{a}v\rangle-\langle u,Av\rangle_{L_{2}}.
\end{align*}
Recalling $Q=Q_{+}^{2}-Q_{-}^{2},$ we may rewrite the latter expression
as follows:
\begin{align*}
\langle Q\tr_{b}u,\tr_{b}v\rangle-\langle Q\tr_{a}u,\tr_{a}v\rangle & =\langle Q_{+}\tr_{b}u+Q_{-}\tr_{a}u,Q_{+}\tr_{b}v+Q_{-}\tr_{a}u\rangle-\\
 & \quad-\langle Q_{-}\tr_{b}u+Q_{+}\tr_{a}u,Q_{-}\tr_{b}v+Q_{+}\tr_{a}u\rangle\\
 & =\langle F_{1}u,F_{1}v\rangle-\langle F_{2}u,F_{2}v\rangle.
\end{align*}
It remains to show that $F\colon H^{n}(a,b)^{d}\to\R^{nd}\times\R^{nd}$
is onto. First we note that $\left(\begin{array}{c}
\tr_{b}\\
\tr_{a}
\end{array}\right)\colon H^{n}(a,b)^{d}\to\R^{nd}\times\R^{nd}$ is onto by \prettyref{lem:trace_onto}. Then $F$ can be expressed
by 
\[
F=\left(\begin{array}{cc}
Q_{+} & Q_{-}\\
Q_{-} & Q_{+}
\end{array}\right)\left(\begin{array}{c}
\tr_{b}\\
\tr_{a}
\end{array}\right)\colon H^{n}(a,b)^{d}\to\R^{nd}\times\R^{nd},
\]
and hence, it suffices to prove that the matrix $\left(\begin{array}{cc}
Q_{+} & Q_{-}\\
Q_{-} & Q_{+}
\end{array}\right)$ is invertible. For doing so, let $(x,y)\in\R^{nd}\times\R^{nd}$
with $\left(\begin{array}{cc}
Q_{+} & Q_{-}\\
Q_{-} & Q_{+}
\end{array}\right)\left(\begin{array}{c}
x\\
y
\end{array}\right)=0,$ that is 
\begin{align*}
Q_{+}x+Q_{-}y & =0,\\
Q_{-}x+Q_{+}y & =0.
\end{align*}
Since $Q_{+}$ and $Q_{-}$ attain values in the orthogonal spaces
$E_{+}$ and $E_{-}$ respectively, it follows that $Q_{+}x=Q_{-}y=0=Q_{-}x=Q_{+}y.$
Hence, also $Qx=\left(Q_{+}^{2}-Q_{-}^{2}\right)x=0=Qy$ and since
$Q$ is invertible, we infer that $x=y=0.$ Thus $\left(\begin{array}{cc}
Q_{+} & Q_{-}\\
Q_{-} & Q_{+}
\end{array}\right)$ is invertible and the assertion follows. 
\end{proof}
As a direct consequence of the latter result and our main result \prettyref{thm:main},
we arrive at the following characterisation of m-accretive restrictions
of $A$. 
\begin{cor}
\label{cor:pH_unweighted}Let $A$ be as in \prettyref{eq:A} and
$B\subseteq A.$ Then $B$ is m-accretive on $L_{2}(a,b)^{d}$ if
and only if there exists a contraction 
\[
g\colon\R^{nd}\to\R^{nd}
\]
such that 
\[
\dom(B)=\{u\in H^{n}(a,b)^{d}\,;\,g(Q_{+}\tr_{b}u+Q_{-}\tr_{a}u)=Q_{-}\tr_{b}u+Q_{+}\tr_{a}u\}.
\]
\end{cor}

We summarise the findings of this subsection in the following theorem,
which is an immediate consequence of \prettyref{cor:pH_unweighted}
and \prettyref{lem:congruence} (b).
\begin{thm}
\label{thm:pH_m-accretive}Let $n,d\in\N$ and $P_{0},\ldots,P_{n}\in\R^{d\times d}$
with $P_{j}^{\top}=(-1)^{j+1}P_{j}$ and $\det P_{n}\ne0.$ Moreover,
let $\mathcal{H}\in L_{\infty}(a,b;\R^{d\times d})$ be such that
$\mathcal{H}(x)$ is symmetric for almost every $x\in[a,b]$ and there
exists $c>0$ such that 
\[
\langle\mathcal{H}(x)v,v\rangle_{\R^{d}}\geq c\|v\|_{\R^{d}}^{2}\quad(v\in\R^{d},x\in[a,b]\text{ a.e.}).
\]
Consider the following operator 
\[
A\colon\dom(A)\subseteq H\to H,\quad Au=\sum_{k=0}^{n}P_{k}\partial^{k}(\mathcal{H}u),
\]
where $H=L_{2}(a,b)^{d}$ equipped with the inner product 
\[
\langle u,v\rangle_{H}\coloneqq\int_{a}^{b}\langle\mathcal{H}(x)u(x),v(x)\rangle\d x\quad(u,v\in L_{2}(a,b)^{d})
\]
and 
\[
\dom(A)\coloneqq\{u\in H\,;\,\mathcal{H}u\in H^{n}(a,b)^{d}\}.
\]
Moreover, let $B\subseteq A.$ Then the following statements are equivalent:

\begin{enumerate}[(i)]

\item $B$ is m-accretive on $H$,

\item There exists a contraction 
\[
g\colon\R^{nd}\to\R^{nd},
\]
such that 
\[
\dom(B)=\{u\in\dom(A)\,;\,g(Q_{+}\tr_{b}\mathcal{H}u+Q_{-}\tr_{a}\mathcal{H}u)=Q_{-}\tr_{b}\mathcal{H}u+Q_{+}\tr_{a}\mathcal{H}u\},
\]
where $\tr_{a},\tr_{b}$ are given by \prettyref{eq:trace} and $Q_{+}$
and $Q_{-}$are defined in \prettyref{eq:Q_+}. 

\end{enumerate}
\end{thm}

This result is well-known for linear restrictions, although it is
mostly presented in a different manner. To see the connection, we
recall a result from \cite{Trostorff_Waurick_2023}.
\begin{lem}[{\cite[Lemma 2.3]{Trostorff_Waurick_2023}}]
\label{lem:Moppi+ich} Let $W\in\R^{nd\times2nd}.$ Then the following
statements are equivalent

\begin{enumerate}[(i)]

\item $W$ has rank $nd$ and satisfies 
\[
W\left(\begin{array}{cc}
-Q & Q\\
1 & 1
\end{array}\right)^{-1}\left(\begin{array}{cc}
0 & 1\\
1 & 0
\end{array}\right)\left(W\left(\begin{array}{cc}
-Q & Q\\
1 & 1
\end{array}\right)^{-1}\right)^{\ast}\geq0.
\]

\item There exists a matrix $M\in\R^{nd\times nd}$ with $\|M\|\leq1$
and an invertible matrix $K\in\R^{nd\times nd}$ such that 
\[
W=K\left(\begin{array}{cc}
Q_{-}-MQ_{+} & Q_{+}-MQ_{-}\end{array}\right).
\]

\end{enumerate}
\end{lem}

Using this result, we obtain the following variant of \prettyref{thm:pH_m-accretive}
in the linear case, which can be found e.g. in \cite[Theorem 3.3.6]{Augner2016}.
\begin{prop}
In the situation of \prettyref{thm:pH_m-accretive} let $B\subseteq A$
be linear. Then the following statements are equivalent:

\begin{enumerate}[(a)]

\item $B$ is m-accretive on $H$,

\item there exists a matrix $M\in\R^{nd\times nd}$ with $\|M\|\leq1$
such that 
\[
\dom(B)=\{u\in\dom(A)\,;\,M(Q_{+}\tr_{b}\mathcal{H}u+Q_{-}\tr_{a}\mathcal{H}u)=Q_{-}\tr_{b}\mathcal{H}u+Q_{+}\tr_{a}\mathcal{H}u\},
\]

\item there exists a matrix $W\in\R^{nd\times2nd}$ with full rank
and 
\begin{equation}
W\left(\begin{array}{cc}
-Q & Q\\
1 & 1
\end{array}\right)^{-1}\left(\begin{array}{cc}
0 & 1\\
1 & 0
\end{array}\right)\left(W\left(\begin{array}{cc}
-Q & Q\\
1 & 1
\end{array}\right)^{-1}\right)^{\ast}\geq0\label{eq:pos_def_W}
\end{equation}
such that 
\[
\dom(B)=\left\{ u\in\dom(A)\,;\,W\left(\begin{array}{c}
\tr_{b}\mathcal{H}u\\
\tr_{a}\mathcal{H}u
\end{array}\right)=0\right\} .
\]

\end{enumerate}
\end{prop}

\begin{proof}
The equivalence of (i) and (ii) is a direct consequence of \prettyref{thm:pH_m-accretive},
see also \prettyref{prop:linear} (b). For the implication (ii) $\Rightarrow$(iii),
we set $W\coloneqq\left(\begin{array}{cc}
Q_{-}-MQ_{+} & Q_{+}-MQ_{-}\end{array}\right).$ Then $W$ satisfies \prettyref{eq:pos_def_W} by \prettyref{lem:Moppi+ich}
and clearly
\[
M(Q_{+}\tr_{b}\mathcal{H}u+Q_{-}\tr_{a}\mathcal{H}u)=Q_{-}\tr_{b}\mathcal{H}u+Q_{+}\tr_{a}\mathcal{H}u\Leftrightarrow W\left(\begin{array}{c}
\tr_{b}\mathcal{H}u\\
\tr_{a}\mathcal{H}u
\end{array}\right)=0.
\]
For (iii) $\Rightarrow$(ii), we note that there exists $M,K\in\R^{nd\times nd}$
with $K$ invertible and $\|M\|\leq1$ such that 
\[
W=K\left(\begin{array}{cc}
Q_{-}-MQ_{+} & Q_{+}-MQ_{-}\end{array}\right)
\]
by \prettyref{lem:Moppi+ich}. Since $\ker W=\ker\left(\begin{array}{cc}
Q_{-}-MQ_{+} & Q_{+}-MQ_{-}\end{array}\right),$ the equality for $\dom(B)$ follows as above.
\end{proof}


\begin{thebibliography}{10}

\bibitem{Augner2016}
B.~Augner.
\newblock {\em Stabilisation of infinite-dimensional port-Hamiltonian systems
  via dissipative boundary feedback}.
\newblock PhD thesis, Bergische Universit\"at Wuppertal, 2016.
\newblock URL:
  \url{https://elekpub.bib.uni-wuppertal.de/urn/urn:nbn:de:hbz:468-20160719-090307-4}.

\bibitem{Engel_Nagel2000}
K.-J. Engel and R.~Nagel.
\newblock {\em One-parameter semigroups for linear evolution equations}, volume
  194 of {\em Grad. Texts Math.}
\newblock Berlin: Springer, 2000.

\bibitem{Gorbachuk1991}
V.~I. Gorbachuk and M.~L. Gorbachuk.
\newblock {\em Boundary value problems for operator differential equations.
  {Transl}. from the {Russian}.}, volume~48 of {\em Math. Appl., Sov. Ser.}
\newblock Dordrecht etc.: Kluwer Academic Publishers, exp. and rev. translation
  edition, 1991.

\bibitem{JacobZwart2012}
B.~Jacob and H.~J. Zwart.
\newblock {\em Linear port-{H}amiltonian systems on infinite-dimensional
  spaces}, volume 223 of {\em Operator Theory: Advances and Applications}.
\newblock Birkh\"{a}user/Springer Basel AG, Basel, 2012.
\newblock Linear Operators and Linear Systems.

\bibitem{Kirszbraun_1934}
M.~D. Kirszbraun.
\newblock {\"U}ber die zusammenziehenden und {Lipschitzschen}
  {Transformationen}.
\newblock {\em Fundam. Math.}, 22:77--108, 1934.

\bibitem{Komura1967}
Y.~{Komura}.
\newblock {Nonlinear semigroups in Hilbert space}.
\newblock {\em {J. Math. Soc. Japan}}, 19:493--507, 1967.

\bibitem{Minty1962}
G.~J. {Minty}.
\newblock {Monotone (nonlinear) operators in Hilbert space}.
\newblock {\em {Duke Math. J.}}, 29:341--346, 1962.

\bibitem{Picard2009}
R.~Picard.
\newblock {A structural observation for linear material laws in classical
  mathematical physics}.
\newblock {\em {Mathematical Methods in the Applied Sciences}}, 32:1768--1803,
  2009.

\bibitem{Picard_Trostorff_2023}
R.~Picard and S.~Trostorff.
\newblock M-accretive realisations of skew-symmetric operators.
\newblock {\em J. Oper. Theory}, to appear.
\newblock arXiv:2207.04824.

\bibitem{PTW2023}
R.~H. Picard, S.~Trostorff, B.~Watson, and M.~Waurick.
\newblock A structural observation on port-{Hamiltonian} systems.
\newblock {\em SIAM J. Control Optim.}, 61(2):511--535, 2023.

\bibitem{Schmudgen2012}
K.~Schm{\"u}dgen.
\newblock {\em Unbounded self-adjoint operators on {Hilbert} space}, volume 265
  of {\em Grad. Texts Math.}
\newblock Dordrecht: Springer, 2012.

\bibitem{SSVW2015}
C.~Schubert, C.~Seifert, J.~Voigt, and M.~Waurick.
\newblock Boundary systems and (skew-)self-adjoint operators on infinite metric
  graphs.
\newblock {\em Math. Nachr.}, 288(14-15):1776--1785, 2015.

\bibitem{Schwartz1969}
J.~T. Schwartz.
\newblock {\em Nonlinear functional analysis}.
\newblock Notes on Mathematics and its Applications. Gordon and Breach Science
  Publishers, New York-London-Paris, 1969.
\newblock Notes by H. Fattorini, R. Nirenberg and H. Porta, with an additional
  chapter by Hermann Karcher.

\bibitem{STW2022}
C.~Seifert, S.~Trostorff, and M.~Waurick.
\newblock {\em Evolutionary equations. {Picard}'s theorem for partial
  differential equations, and applications}, volume 287 of {\em Oper. Theory:
  Adv. Appl.}
\newblock Cham: Birkh{\"a}user, 2022.

\bibitem{Trostorff_Waurick_2023}
S.~Trostorff and M.~Waurick.
\newblock Characterisation for Exponential stability of port-hamiltonian systems.
\newblock Technical report, CAU Kiel and TU BA Freiberg, 2023.
\newblock arXiv:2201.10367.

\bibitem{Villegas2007}
J.~A. Villegas.
\newblock {\em A port-Hamiltonian Approach to Distributed Parameter Systems}.
\newblock PhD thesis, Universiteit Twente, 2007.
\newblock URL:
  \url{https://research.utwente.nl/files/6041264/thesis_Villegas.pdf}.

\bibitem{Waurick_Wegner_2018}
M.~Waurick and S.-A. Wegner.
\newblock Some remarks on the notions of boundary systems and boundary
  triple(t)s.
\newblock {\em Math. Nachr.}, 291(16):2489--2497, 2018.

\bibitem{Wegner2017}
S.-A. Wegner.
\newblock Boundary triplets for skew-symmetric operators and the generation of
  strongly continuous semigroups.
\newblock {\em Anal. Math.}, 43(4):657--686, 2017.

\end{thebibliography}
\end{document}